\documentclass[draft]{amsart}

\usepackage{tikz}
\usetikzlibrary{calc}

\DeclareMathOperator{\Rp}{Re}
\DeclareMathOperator{\Ip}{Im}

\newtheorem{theorem}{Theorem}[section]
\newtheorem{corollary}[theorem]{Corollary}
\newtheorem{lemma}[theorem]{Lemma}

\newcommand{\wt}[1]{\widetilde{#1}}

\begin{document}

\title[Norm of the Backward Shift]{
Bounds on the Norm of the Backward Shift and Related 
Operators in 
Hardy and Bergman Spaces}

\author{Timothy Ferguson}

\begin{abstract}
We study bounds for the backward shift operator $f \mapsto (f(z)-f(0))/z$ and the related operator $f \mapsto f - f(0)$ on Hardy and Bergman spaces of analytic and harmonic functions.  If $u$ is a real valued harmonic function, we also find a sharp bound on $M_1(r,u-u(0))$ in terms of $\|u\|_{h^1}$, where $M_1$ is the integral mean with $p=1$.
\end{abstract}

\maketitle

\section{Introduction}

\newcommand{\szop}{\mathcal{B}}
\newcommand{\bsop}{B}

For a space of continuous functions in the unit 
disc with bounded point evaluation, we can 
consider the operator $\szop$ defined by 
$\szop(f) = f - f(0)$. For a space of analytic 
functions in the unit disc, it also makes sense 
to consider the backward shift operator 
$\bsop$ defined by $\bsop(f) = [f-f(0)]/z$.  
On $H^2$, the backward shift operator is the adjoint of the 
forward shift operator given by $Sf(z) = z f(z)$.  
Both $S$ and $\bsop$ have been extensively studied in the 
literature.  

In this article, we study the norms of these operators on various spaces. 
For Hardy spaces, both $\bsop$ and $\szop$ have the same 
norm, and since $|f(0)| \leq \|f\|$, it is 
clear that the norm is at most $2$.  However, 
we are not aware of anywhere in the literature that 
discusses this question further, beyond the 
observation that the norms of both $\bsop$ and 
$\szop$ are exactly $1$ for $H^2$.  One can use the above 
facts and interpolation to
show that 
$\|\bsop\|_{H^p} \leq 2^{(2-p)/p}$ if $1 \leq p \leq 2$
and $\|\bsop\|_{H^p} \leq 2^{(p-2)/p}$ for 
$2 \leq p \leq \infty$.  However, this does not 
settle the question of whether the norm on 
$H^1$ and $H^\infty$ is less than $2$. 
We 
prove that $\|\bsop\|_{H^\infty} = 2$ but that 
$\|\bsop\|_{H^1} < 2$.  In fact, we prove that 
$\|\bsop\|_{H^1} \leq 1.71$.  

We also study bounds for $\bsop$ and $\szop$ on other spaces. 
Let $A^p$ denote the Bergman space of the unit disc with normalized 
area measure and let $a^p_{\mathbb{R}}$ be the real harmonic Bergman space on 
the unit disc with normalized area measure.  
We show that 
$\|\szop\|_{A^p} \leq \|\szop\|_{H^p}$. 
This also implies that 
$\|\bsop\|_{A^p} \leq 2 \|\szop\|_{H^p}$.  
We also show that $\|\szop\|_{a^1_{\mathbb{R}}} \leq 1.84$.  

Lastly, we consider the operator $\szop_r$ 
from the real valued harmonic Hardy space $h^1_{\mathbb{R}}$ onto 
$L^1(\partial \mathbb{D})$ defined by
\[ 
u \mapsto u(re^{i\theta}) - u(0)
\]
This is the operator that maps 
a harmonic function $u$ in $h^1_{\mathbb{R}}$ to the 
restriction of $u - u(0)$ to the circle of radius $r$. 
We show that this operator has norm 
$2 - \tfrac{4}{\pi}\arccos(r)$.  Furthermore, the maximum in the 
definition of norm is attained, for example by the Poisson kernel.

It turns out that the question of bounds for 
$\bsop$ and $\szop$ on subspaces of $L^1$ are related to 
questions about concentrations of 
functions.  
Consider for example the space $h^1_{\mathbb{R}}$.  Let $f_n(z)$ be the 
function in $h^1_{\mathbb{R}}$ with boundary values given by 
\begin{equation}\label{eq:cutoff-def}
f_n(e^{i \theta}) = \begin{cases} 
   \pi n &\text{ if $-1/n \leq \theta \leq 1/n$,}\\
     0 &\text{otherwise.}
\end{cases}
\end{equation}
  Then $f_n(0) = 1$ and 
$\|f_n(z) - f_n(0)\|_{h^1_{\mathbb{R}}} = 2 - (4/n)$.  
This shows that the norm of $\szop$ on $h^1_{\mathbb{R}}$ is $2$.  
Notice that the boundary functions 
of the $f_n$ have most of their $L^1$ norm concentrated on sets of 
small measure.  Moreover, their sign does not oscillate on these sets.  
Contrast these functions with the functions 
$g_n$ with boundary values given by 
\[
g_n(e^{i \theta}) = \begin{cases} 
   -\pi n &\text{ if $-1/n \leq \theta < 0$}\\
   \pi n &\text{ if $0 \leq \theta \leq 1/n$}\\
     0 &\text{otherwise}
\end{cases}.
\]
The $g_n$ also have boundary values concentrated on sets of small measure, 
but the sign of their boundary values oscillates, which allows 
$g_n(0)$ to be $0$.  Thus $\|g(z)-g(0)\|_{h^1_{\mathbb{R}}} = 1$. 
Roughly speaking, for $\szop f$ to have large norm, 
$f$ should have most of its mass concentrated on a set of small measure, 
and its sign should not oscillate much.  

For a related example, we define the 
Poisson kernel by 
\[
P_{r}(e^{i t}) = \frac{1-r^2}{1-2r\cos(t) + r^2}.
\]
If $z = re^{it}$, then  $P_{r}(t)$ is a harmonic function of 
$z$, which we may denote here by $P(z)$.  
It is not difficult to see that 
$P(0) = 0$ and 
$\|P - 1\|_{h^1_{\mathbb{R}}} = 2$
This is related to the 
fact that the Poisson kernel is a Poisson integral of a point mass,
so that $P$ is as concentrated as possible in some sense. 
This example also hints at that the fact that 
$\|\szop\|_{H^1} < 2$, because of the fact that the analytic 
completion of the Poisson kernel is not in $H^1$.  Thus, we might
expect that there is some limit to the concentration of boundary 
values of $H^1$ functions.  Related to this is the fact that 
it is not difficult to find 
an analytic function in $H^1$ that is large only on a set of small 
measure - for example $c_n (1+z)^n/2^n$, where 
$c_n$ is chosen so that the function has norm $1$.  However, such 
functions oscillate in sign near the points where they are large. 

In order to formalize the above observations, 
we prove two different theorems about 
concentration of of analytic functions 
on sets 
of small measure.  
Suppose $\|f\|_{H^1} = 1$.  
In the first theorem, we prove 
that if $A \subset \mathbb{T}$ and 
$\int_A \Rp f(0) \, d\theta/(2\pi) > 1 - \epsilon$ for 
small enough $\epsilon$, then $m(A)$ cannot 
be too small.  
The second says that if 
$\|f\|_{L^1(A)} > 1 - \epsilon$ for some small enough set 
$A$ and for small enough $\epsilon$, then $f(0)$ cannot be too large. 
We provide two proofs that 
$\|\szop\|_{H^1} < 2$, where each proof uses one of the above theorems.  

\section{Bounds for Hardy spaces}
Let $0 < p < \infty$.
For a continuous function $f$ in the unit 
disc $\mathbb{D}$, we define the $p^{\textrm{th}}$ integral mean of $f$ at radius $r$ by 
\[
M_p(r,f) = \left( \frac{1}{2\pi} \int_0^{2\pi} |f(re^{i\theta})|^p \right)^{1/p}.
\]
We define $M_\infty(r,f) = \sup_{0 \leq \theta < 2\pi} |f(re^{i\theta})|$. 

We define $H^p$ to be the space of 
analytic functions in the unit disc 
such that $\|f\|_{H^p} = \sup_{0 \leq r < 1} M_p(r,f) < \infty$. 
We define $h^p_{\mathbb{R}}$ to be the space of real valued harmonic 
functions in the unit disc such that 
$\|f\|_{H^p} = \sup_{0 \leq r < 1} M_p(r,f) < \infty$.  Note that $M_p(r,f)$ is increasing 
for $0 < p \leq \infty$ if $f$ is analytic 
and for $1 \leq p \leq \infty$ if $f$ is 
harmonic (see \cite[Theorems 1.5 and 1.6]{D_Hp}). 
Functions in $H^1$ have radial limits almost everywhere on the boundary 
of the unit disc, and they are uniquely determined by their boundary 
value functions.  In fact, the norm of an $H^1$ function is equal to the 
$L^1$ norm of its boundary function. In contrast to this, $h^1_{\mathbb{R}}$ functions, 
even though they have radial limits almost everywhere, are not uniquely 
determined by their boundary values.  However, they can be written as 
convolutions of Poisson kernels with finite Radon measures 
(see \cite{D_Hp}). 

We define $A^p$ to be the subspace of 
$L^p$ of the unit disc (with normalized 
area measure) consisting of analytic 
functions.  Let $a^p_{\mathbb{R}}$ be the (real) 
subspace of $L^p$ consisting of real 
valued harmonic functions.

In this section we discuss bounds for the operators 
$\bsop$ and $\szop$ on $H^1$ and also on 
$H^\infty$. 
We begin with two theorems.  
The first roughly says that an analytic function in 
$H^1$ that has most of the mass of its boundary 
value function concentrated on a small set must 
show an appreciable degree of cancellation if the 
function is integrated over the set. The second 
theorem is similar, but deals instead with the 
integral of the function over all of 
the unit circle. 

\begin{theorem}\label{thm:concentrated-rp-set-bound} Suppose that $\|f\|_{H^1} = 1$ and that for some set 
$A \subset \mathbb{T}$, we have $\int_A \Rp f \, dt/2\pi \geq 1-\epsilon$ for 
$\epsilon < 1/4$.  
Then 
\[
m(A) \ge \max_{0 < \gamma < 1} 
   \frac{\log(\gamma + \epsilon) - \log(1 - 2\epsilon)}{\log \gamma}
\] 
\end{theorem}
\begin{proof}
Let 
\[ 
F(z) = \exp \left\{ \frac{1}{2\pi} 
   \int_0^{2\pi} \frac{e^{it} + z}{e^{it}-z} \log \psi(t) \, dt \right\}
\]
where 
\[
\psi(t) = \begin{cases} \gamma \text{ on $A$} \\
                        1 \text{ on $A^c$}
\end{cases}
\]
and $0 < \gamma < 1$.  Then $F(0)$ is real and 
$\log(F(0))  = \log(\gamma) \cdot m(A)$, and thus 
$F(0) = \gamma^{m(A)}$. 
Also, $|F(e^{i\theta})| = 1$ a.e.\ if $e^{i\theta} \in A^c$, and 
$|F(e^{i\theta})| = \gamma$ a.e.\ if $e^{i\theta} \in A$ a.e. 
Now note that 
\[
\left| \int_0^{2\pi} f(e^{it}) \, \frac{dt}{2\pi} \right| 
\geq \left| \int_A \Rp f(e^{it}) \, \frac{dt}{2\pi} \right| - 
\int_{A^c} |f(e^{it})| \frac{dt}{2\pi}
\geq (1 - \epsilon)-\epsilon.
\]
Thus
$|f(0)| \ge 1 - 2\epsilon$
and
\[
|F(0) \cdot f(0)| \ge \gamma^{m(A)} (1 - 2\epsilon).
\]

But also
\[
|F(0) \cdot f(0)| = 
\left|\int F\cdot f (e^{i\theta}) \frac{dt}{2\pi} \right| 
\le \gamma \int_{A} |f| \frac{dt}{2\pi} + 
\int_{A^c} |f| \frac{dt}{2\pi}
   \le \gamma + \epsilon.
\]

Thus 
\[
\gamma^{m(A)} (1 - 2\epsilon)
\le 
\gamma  + \epsilon.
\]

But this means 
\[ 
m(A) \ge \frac{\log(\gamma + \epsilon) - \log(1 - 2\epsilon)}{\log \gamma}.
\]
To see the maximum of this quantity for $0 < \gamma < 1$ 
is attained, notice that the expression on the right 
of the above inequality approaches $0$ as $\gamma \rightarrow 0^+$ and 
approaches $-\infty$ as $\gamma \rightarrow 1^-$, but is positive for 
$\gamma = \epsilon$. 
\end{proof}

In the next theorem we let $m$ denote normalized arc length measure.  
\begin{theorem}\label{thm:concentrated-abs-f0-bound}
Suppose that $f \in H^p$ and that $\|f\|_{H^p}=1$.  Furthermore, 
suppose that $\|f\|_{L^1(E)} \ge 1 - \epsilon$ for some set 
$E \subset \mathbb{T}$ where $m(E) \le \delta$ 
 and 
$0 < \epsilon, \delta < 1/2$.  Then 
\[
|f(0)| \le 
\left(\frac{1-\epsilon}{\delta}\right)^{\delta}
\left( \frac{\epsilon}{1-\delta} \right)^{1-\delta}.
\]
\end{theorem}
\begin{proof}
Let $m(E) = \delta_0$ and $\|f\|_{L^1(E)} = 1 - \epsilon_0$. 
Writing $f$ as the product of an outer function and an inner function 
(see \cite{D_Hp})  
shows that it suffices to assume that $f$ is outer.  Then we may 
assume that 
\[
f(z) = \exp\left\{ \frac{1}{2\pi}\int_0^{2\pi} 
              \frac{e^{it}+z}{e^{it}-z} \log \psi(t) \, dt
\right\}
\]
for some nonnegative function $\psi \in L^p$ such that 
$\log \psi \in L^1$.  Also $|f(e^{it})| = \psi(t)$ a.e. Note that 
$f(0) = \exp ( \frac{1}{2\pi} \int_0^{2\pi} \log \psi(t) \, dt ).$
Now, Jensen's inequality shows that 
\[
\begin{split}
\exp \left( \int_{E} \log \psi(t) \, \frac{dt}{2\pi} \right) 
&= 
\exp \left( \int_{E} m(E) \log \psi(t) \, \frac{dt}{2\pi m(E)} \right) 
\\
&= 
 \left\{
   \exp \left( \int_{E}  \log \psi(t) \, \frac{dt}{2\pi m(E)} \right) 
 \right\}^{m(E)}
\\ &\le 
\left\{\int_{E} \psi(t) \, \frac{dt}{2\pi m(E)}\right\}^{m(E)} \\
& = \left(\frac{1-\epsilon_0}{\delta_0}\right)^{\delta_0}.
\end{split}
\]
A similar calculation for the set $E^c$ shows that 
\[
\exp \left( \int_{E^c} \log \psi(t) \frac{dt}{2\pi} \right) 
\le 
\left( \frac{\epsilon_0}{1-\delta_0} \right)^{1-\delta_0}.
\]
Putting this together gives
\[
f(0) = \exp \left( \int_{E \cup E^c} \log \psi(t) \, \frac{dt}{2\pi} \right) \le 
\left(\frac{1-\epsilon_0}{\delta_0}\right)^{\delta_0}
\left( \frac{\epsilon_0}{1-\delta_0} \right)^{1-\delta_0}.
\]

Since the function $(1 - \epsilon_0)^x \epsilon_0^{1-x}$ is increasing, 
it follows that 
$(1 - \epsilon_0)^{\delta_0} \epsilon_0^{1-\delta_0} 
\leq (1 - \epsilon_0)^{\delta} \epsilon_0^{1-\delta}.$ 
Since the function $x(1-x)$ is increasing for $0 < x < 1/2$, 
we have 
\[
(1 - \epsilon_0)^{\delta} \epsilon_0^{1-\delta} = 
(1 - \epsilon_0)^{\delta} \epsilon_0^{\delta} \epsilon_0^{1-2\delta}
\leq
(1 - \epsilon)^{\delta} \epsilon^{\delta} \epsilon^{1-2\delta} = 
(1 - \epsilon)^{\delta} \epsilon^{1-\delta}.
\]
Since the function $x^x (1-x)^{1-x}$ is decreasing for 
$0 \le x < 1/2$, we have 
\[
\delta_0^{-\delta_0} (1-\delta_0)^{\delta_0-1} \le 
\delta^{-\delta} (1-\delta)^{\delta-1}.
\]
if 
$0 < \delta_0 \le \delta < 1/2$.  Putting this together gives 
the result. 
\end{proof}

We now use Theorem 
\ref{thm:concentrated-rp-set-bound} to 
bound $\|\bsop\|_{H^1}$. 
\begin{theorem} 
The norm of $\szop$ on the Hardy space 
$H^1$ is at most $1.952396$.
\end{theorem}
\begin{proof}
Let $m$ denote Lebesgue measure divided by $2\pi$ on the unit circle. 
Suppose that $\|f\|_{H^1} = 1$. 
Without loss of generality, we may assume that $f(0) > 0$.  
Suppose that
$\|f - f(0)\|_{H^1} > 2 - \alpha$ for some $0 < \alpha < 1/2$.  We will 
show that this leads to a contradiction for small enough $\alpha$.  
Note that $f(0) > 1 - \alpha$ and $|f(0)| \le 1$. 

Now consider $u = \Rp f$ and $v = \Ip f$.  
Let $0 < \beta < 1/2$.  Define 
$A = \{e^{i\theta} : u > \beta\}$ and 
$B = \{e^{i\theta} : u \le \beta\}$.  
Now if $u(e^{i\theta}) \ge u(0)$ we have 
$|u(e^{i\theta}) - u(0)| \le |u(e^{i\theta})|$ 
so $|f(e^{i\theta}) - f(0)| \le |f(e^{i\theta})|$ since 
$f(0) = u(0)$ is real. 

However if 
$ \beta < u(e^{i\theta}) < u(0)$ we have 
$|f(e^{i\theta}) - f(0)| = |u(0) - u(e^{i\theta}) + iv(e^{i\theta})|  
\le 1 - \beta + |f(e^{i\theta})|$. 
So if $e^{i\theta} \in A$ we have 
$|f(e^{i\theta}) - f(0)| \le |f(e^{i\theta})| + 1 - \beta$. 
Thus
\[
\int_A |f - f(0)| \, dm \le \int_A |f| \, dm + (1 - \beta) m(A). 
\]

And if $e^{i\theta} \in B$, we have 
$|f(e^{i\theta}) - f(0)| \le |f(e^{i\theta})| + 1$.
Thus 
\[
\int_B |f - f(0)| \, dm \le \int_B |f| \, dm + m(B) .
\]
Therefore,
\[
\begin{split}
2-\alpha < \int |f - f(0)| \, dm &\le \int |f| \, dm + (1 - \beta)m(A) 
+ m(B) 
\\
 &=  1  + (1 - \beta)m(A) + m(B) 
\end{split}
\] 

And therefore 
$
(1 - \beta) m(A) + m(B) > 1 - \alpha$.  But $m(A) + m(B) = 1$ so 
$ -\beta m(A) \ge - \alpha$ so 
\[
m(A) \le \frac{\alpha}{\beta}.
\]

But it is also clear that 
$\int_A u \, dm + \beta > \int u \, dm > 1 - \alpha$.
By Theorem \ref{thm:concentrated-rp-set-bound} we have 
\[
\alpha / \beta \ge 
 \max_{0 < \gamma < 1} 
   \frac{\log(\gamma + (\alpha + \beta)) - \log(1 - 2(\alpha + \beta))}
    {\log \gamma}
\]
as long as $\alpha + \beta < 1/4$. 
However, this is false for 
$\alpha = .047604$ and $\beta = .127079$, as can be seen 
by taking $\gamma = .104634$. 

\end{proof}

Similarly to the above theorem, we now use 
Theorem 
\ref{thm:concentrated-abs-f0-bound} to bound 
$\|\bsop\|_{H^1}$. 
\begin{theorem} The norm of the backward shift operator on 
$H^1$ is at most $1.7047$. 
\end{theorem}
\begin{proof}
Suppose that $\|f-f(0)\|_{H^1} > 2 - \alpha$ for $0 < \alpha < 1/2$.  
Then $|f(0)| > 1 - \alpha$ 
and we may assume without loss of generality that $f(0)$ is positive. 
Now consider $u = \Rp f$ and $v = \Ip f$.  
Let $0 < \beta < 1/2$.  Define 
$A = \{e^{i\theta} : u > \beta\}$ and 
$B = \{e^{i\theta} : u \le \beta\}$.  
By the reasoning in the proof of the previous theorem, 
$m(A) \le \alpha / \beta$ and 
$\int_A |f| \, dm \geq \int_A u \, dm > 1 - \alpha - \beta$. 
So by Theorem \ref{thm:concentrated-abs-f0-bound}, 
\[
1 - \alpha < |f(0)| \le 
 \left( \frac{1-(\alpha + \beta)}{\alpha/\beta} \right)^{\alpha/\beta}
\left( \frac{\alpha + \beta}{1 - \alpha/\beta} \right)^{1 - (\alpha/\beta)}. 
\]
However, the above inequality is false if 
$\alpha = .295302$ and $\beta = .476286$. 
\end{proof}

In contrast to the case with 
$H^1$, we show that the norm of the backward 
shift operator on $H^2$ is exactly $2$. 
\begin{theorem}\label{thm:hinfinity-bound}
The norm of the backwards shift operator on $H^\infty$ is 
$2$.
\end{theorem}
\begin{proof}
It is clear that 
$\|\bsop\|_{H^\infty} \leq 2$.  
Now consider the region 
\[
R = \{z: |z-(1-\epsilon)|<\epsilon\} \cup 
     \Delta(1-\epsilon - \epsilon e^{-i\delta}, 
                   1-\epsilon - \epsilon e^{i\delta}, -1)
\]
where $\Delta(a,b,c)$ denotes the triangle with vertices $a$, $b$, and 
$c$ (see Figure \ref{fig:hinfty1} on page \pageref{fig:hinfty1}). 
Let $f$ be the conformal map from $\mathbb{D}$ to $R$ that sends 
the origin to $1-\epsilon$ and that if extended to the boundary sends 
 $1$ to $1$. 

Now $f$ attains the value $-1$ on the boundary of the disc.  
In fact
$f(-1) = -1$ since the uniqueness of the conformal map
shows that $\overline{f(\overline{z})} = f(z)$.  So 
\[
|f(-1) - f(0)| = |-1 - (-1 - \epsilon)| = 2 - \epsilon. 
\]
So 
$\|f-f(0)\|_\infty \geq 2 - \epsilon$.  Since this holds for any 
$\epsilon > 0$, the theorem is proven. 

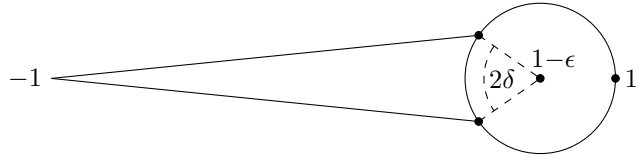
\begin{figure}
\begin{tikzpicture}[scale=5.0]

\coordinate (cent) at (.8,0);
\coordinate (A) at ($(cent) + (145:.2cm)$);
\coordinate (B) at ($(cent) + (215:.2cm)$);
\coordinate (left) at (-.5,0);
\draw[fill=black] (cent) circle (.01cm) node[anchor=south]{$\ \ \ 1{-}\epsilon$};
\draw (left) node[anchor=east]{$-1$} -- (A) ; 
\draw (left) -- (B);
\draw (cent) circle (.2cm);
\draw[fill=black] (1,0) circle(.01cm)node[anchor=west]{$1$};
\draw[fill=black] (A) circle(.01cm);
\draw[fill=black] (B) circle(.01cm);
\draw[dashed] (A) -- (cent);
\draw[dashed] (B) -- (cent);
\draw ($(cent)+(-.09,0)$) node{$2\delta\ $};
\draw[dashed] ($(cent)+(145:.15cm)$) arc (145:215:.15cm);

\end{tikzpicture}
\caption{The region R from the proof of Theorem \ref{thm:hinfinity-bound}.}
\label{fig:hinfty1}
\end{figure}
\end{proof}

\section{Bergman Spaces}

We first prove a theorem relating the norm of 
$\szop$ on Bergman spaces to its norm on 
Hardy spaces.

\begin{theorem} \label{thm:norm-szop-Ap}
Suppose that the norm of the operator 
$\szop$ is equal to $K$ on $H^p$.  Let 
$\mu$ be a radial weight such that $\mu(\mathbb{D}) < \infty$. 
Then the norm of 
$\szop$ is at most $K$ on the 
Bergman space $A^p(\mu)$. 
\end{theorem}

\begin{proof}
Let $\mu \, dA = 
 \widetilde{\mu}(r) 2r \, dr \, d\theta$.  
Note that 
\(
M_p(r,f-f(0)) \le K M_p(r,f)
\)
by the Hardy space bound applied to the dilation $f_r(z)$.  
Thus
\[
\begin{split}
\|f-f(0)\|_{A^p(\mu)}^p &= \int_0^1 M_p^p(r,f-f(0)) 2r \, d\widetilde{\mu}(r) 
\le K^p \int_0^1 M_p^p(r,f) 2r \, d\widetilde{\mu}(r) 
\\
&= K^p \|f\|_{A^p(\mu)}^p.
\end{split}
\]
\end{proof}

We now prove a similar theorem to the one 
above, but for $\bsop$. This theorem is 
slightly more difficult to prove since 
dividing by $z$ can increase the norm of a 
Bergman space function. 
\begin{theorem} Suppose that the norm of the operator 
$\bsop$ is equal to $K$ on $H^1$.  Let 
$\mu$ be a finite radial weight that is increasing.  Then the norm of 
$\bsop$ is at most $2K$ on the Bergman space $A^1(\mu)$. 
\end{theorem}
\begin{proof}
This theorem follows immediately from 
Theorem \ref{thm:norm-szop-Ap} and the 
following lemma.
\end{proof}

\begin{lemma}
Let $\mu$ be an increasing radial 
measure.  Then
$\|zf\|_{A^1(\mu)} \geq (1/2)\|f\|_{A^1(\mu)}$.
\end{lemma}
\begin{proof}
Let $\mu \, dA = 
 \widetilde{\mu}(r) 2r \, dr \, d\theta$.  
Note that 
\[
\begin{split}
&\phantom{{}={}}\int_0^1 M_1(r,f) 2r  \, \widetilde{\mu}(r) dr \\ &= 
\int_0^{1/2} M_1(r,f) 2r + M_1(1-r,f)2(1-r) \widetilde{\mu}(1-r) \, dr \\ 
&=
\int_0^{1/2} M_1(r,f) 4(1/2)r + M_1(1-r,f)4(1/2)(1-r) \, \widetilde{\mu}(1-r)dr \\
\end{split}
\]
Now note that $M_1(r,f) \leq M_1(1-r,f)$ 
and $\widetilde{\mu}(r) \leq \widetilde{\mu}(1-r)$
for $0 \leq r \leq 1/2$ to see that the last displayed 
expression is at most
\[
\begin{split}
&\int_0^{1/2} M_1(r,f) 4r^2 \widetilde{\mu}(r) 
  + M_1(1-r,f)4(1-r)^2 \,\widetilde{\mu}(1-r) dr \\
=2&\int_0^1 M_1(r,z f) 2r \, \widetilde{\mu}(r) dr.
\end{split}
\]
\end{proof}

Using a different method, we now establish a 
bound for $\szop$ on the real harmonic Bergman space 
$a^1_{\mathbb{R}}$. 
\begin{theorem} The norm of the backward shift on the real harmonic 
Bergman space 
$a^1_{\mathbb{R}}$ is at most $1.835$.  In fact, the same estimate holds on any subspace 
$X$ of $L^1$ with the property that $u \in X$ implies that 
$|u(re^{i\theta})| \leq (1-r)^{-2}$ and the property that the average value 
of any $u \in X$ on circles centered at the origin is constant. 
\end{theorem}
\begin{proof}
Suppose that there is a $u \in a^1_{\mathbb{R}}$ with $\|u\|=1$ and 
$\|u-u(0)\| > 2 - \alpha$, where 
$0 < \alpha < 1/2$.  This implies that 
$u(0) > 1 - \alpha$. 
Without loss of generality assume $u(0) > 0$. 
Choose $\beta$ such that 
$0 < \beta < 1/2$ and define $A = \{z : u > \beta\}$ and 
$B = \{z: u \leq \beta\}$.  

Now, we have that $|u(z)| \leq (1-r)^{-2}$  
(see \cite[Chapter 1, Theorem 1]{D_Ap}). 
Let $A_r = A \cap \{z: |z| = r\}$ and define $B_r$ similarly.  

Let $m$ denote normalized area measure and 
let $m_r$ denote Hausdorff $1$-measure on the circle of radius $r$ 
divided by $2\pi r$. 
Since 
\[\int_{A_r} u \, dm_r + \beta > \int u \, dm_r = u(0) > 1 - \alpha\] we have 
$\int_{A_r} u \, dm_r> 1 - \alpha - \beta$.

Notice that $\int_{B_r} u \, d m_r \leq m_r(B_r) \beta = (1-m_r(A_r))\beta$ and 
$\int_{A_r} u \, d m_r \leq m_r(A_r)/(1-r)^2$.  Since 
$\int_{A_r} u \, d m_r + \int_{B_r} u \, d m_r =  u(0)$, we have that 
\[
m_r(A_r) \geq \frac{u(0)-\beta}{(1-r)^{-2}-\beta},
\]
and thus
\[
m_r(A_r) \geq \frac{1-\alpha-\beta}{(1-r)^{-2}-\beta}.
\]
Therefore
\[
m(A) = \int_0^1 m(A_r) 2r \, dr \geq 
\frac{1-\alpha-\beta}{\beta}\left(
  \frac{1}{2\sqrt{\beta}} \ln \left( 
    \frac{1+\sqrt{\beta}}{1-\sqrt{\beta}} \right) - 1 \right).
\]

However, 
\[
\int_B |u-u(0)|\, dm \leq \int_B |u| \, dm + \int_B u(0)\,dm 
= \int_B |u|\,dm + m(B) u(0).
\]
If $z \in A$ we have $|u(z)-u(0)| \leq |u(z)| + u(0) - 2\beta$.
Therefore, 
\[
\int_A |u - u(0)| \,dm  \leq \int_A |u|\, dm + m(A)u(0) - 2m(A) \beta
\]
and thus
\[
\int |u-u(0)| \, dm \leq \int |u| \, dm + u(0) - 2 m(A) \beta.
\]

This implies that
\[
\int |u-u(0)| - 1 - u(0) \leq -2 m(A) \beta,
\]
so 
\[
m(A) \leq \frac{1}{2\beta} (1 + u(0) - \|u-u(0)\|) \leq \frac{\alpha}{2\beta}.
\]

Therefore
\[
\frac{1-\alpha-\beta}{\beta}\left(
  \frac{1}{2\sqrt{\beta}} \ln \left( 
    \frac{1+\sqrt{\beta}}{1-\sqrt{\beta}} \right) - 1 \right)
\leq
\frac{\alpha}{2\beta}.
\]
Choosing $\beta = .506$ and $\alpha = .165$ gives a contradiction. 

\end{proof}

\section{The norm of the operator 
$\szop_r$.}
Suppose that $f$ and $g$ are functions defined on the 
interval $[a,b]$.  By the convolution of $f$ and $g$, we 
mean the function 
\[
f*g(x) = \int_a^b \widetilde{f}(y) \widetilde{g}(x-y) \, dy
,
\]
where $\widetilde{f}$ and $\widetilde{g}$ are the 
periodic extensions of $f$ and $g$ to the real line. 

We define the operator 
$\szop_{r}: h^1_{\mathbb{R}} \rightarrow h^1_{\mathbb{R}}$ to be the 
operator $f \mapsto (\szop f)_r$, where 
$(\szop f)_r(z) = (\szop f)(rz)$.  
Equivalently 
$\szop_r$ can be thought of as the operator 
obtained by applying $\szop$ and then 
restricting the function obtained to the 
circle centered at the origin with radius 
$r$. In this section, we investigate the 
norm of 
$\szop_r$ and find that the Poisson kernel is 
a solution to the problem of finding a function 
$f \in h^1_{\mathbb{R}}$ of norm $1$ such that 
$\|\szop_r f\|$ is as large as possible.  

In order to proceed we need to prove several 
lemmas.  The first is elementary but 
is surprisingly useful.
\begin{lemma}\label{lemma:alt-f-minus-mu}
Let $f$ be a real function with average $\mu$ and 
let $\nu$ be a finite measure.  Then 
\[
\int |f-\mu| \, d\nu 
= 2 \int_{\{x:f > \mu\}} f - \mu \, d\nu.
\]
\end{lemma}
\begin{proof}
Note that 
$\int_{\{x:f = \mu\}} f - \mu \, d\nu = 0$.  
Since $\int (f-\mu) \, d\nu = 0$, this 
implies that 
$\int_{\{x:f > \mu\}} f - \mu \, d\nu = 
\int_{\{x:f < \mu\}} |f - \mu| \, d\nu$.
\end{proof}

The next theorem basically deals with 
maximizing $\int_{\{x:f > \mu\}} |f-\mu|$, where 
$f$ is itself an average of rearrangements of 
some other function.  However, we present the 
theorem in a discreet form which is easier 
to prove. 
\begin{theorem}\label{thm:matrix-sum-minus-mu}
Suppose we are given an $m \times n$ matrix $A$.  Let 
$\mu$ be a fixed number, and let 
$C_j = \sum_{k=1}^m a_{kj}$.  
Define $D_j = \max(C_j - \mu, 0)$, and let 
$D = \sum_{j=1}^n D_j$.  
Suppose that $D_j \geq D_k$ but $a_{ij} \leq a_{ik}$.  Then we 
do not decrease $D$ by interchanging $a_{ij}$ and $a_{ik}$.
\end{theorem}

\begin{proof}
We may assume without loss of generality that $D_1 \geq D_2$, but that 
$a_{11} \leq a_{12}$.  Let $A'$ be the matrix formed by interchanging 
$a_{11}$ and $a_{12}$. We claim that $D' \geq D$. 

If $D = 0$ then we are done.  If $D_1 > 0$ and $D_2 = 0$, then 
since $D_1' \geq D_1$ and $D_2' \geq 0$, we are done.  Suppose then that 
$D_1 > 0$ and $D_2 > 0$.  If $D_2' > 0$, then also 
$D_1' > 0$ so  
$D_1 + D_2 = D_1' + D_2' = C_1 + C_2 - 2\mu$, so we are done.  
However, suppose that $D_2 > 0$ but $D_2' = 0$.  
Let $B_1$ the the sum of the 
rest of the entries in column $1$, and $B_2$ be the sum of the 
rest of the entries in column $2$.  Then 
$D_1 + D_2 = a_{11} + a_{12} + B_1 + B_2 - 2\mu$ and 
$D_1' + D_2'= a_{12} + B_1 - \mu$.  So we will have 
$D_1 + D_2 \leq D_1' + D_2'$ if $a_{11} + B_2 \leq \mu$.  But this is 
true since $D_2' = 0$.  
\end{proof}

We now have the following lemma, which is a continuous form of the above 
theorem.  We omit the proof, since we prove a more general version in 
\ref{lemma:negative-conv-bound}. 
Note that $P * 1$ is the integral of $P$, and similarly for 
$P * f * 1$ and $P * f$. 
\begin{lemma} \label{cor:nonnegative-conv-bound}
 Suppose that $P$ 
and $f$ are nonnegative integrable functions on 
$[a,b]$ and that $\|f\|_1 = 1$, where $\|\cdot\|_1$ denotes 
the $L^1$ norm with normalized Lebesgue measure.   
Then
$\|P * f - P * f * 1\|_1 \leq \|P - P * 1\|_1 $. 
\end{lemma}

\begin{corollary}
Suppose that $u$ is a nonnegative harmonic function in the unit disc
and that $0 \leq r < 1$.  Then 
\[
\begin{split}
\frac{1}{2\pi}\int_0^{2\pi} |u(re^{i\theta}) - u(0)| \, d\theta
&\leq 
u(0) \frac{1}{2\pi} \int_0^{2\pi} |P_r(e^{i\theta}) - 1| \, d\theta 
\\ &= 
u(0) \left( 2 - \frac{4}{\pi}\arccos(r) \right).
\end{split}
\]
where $P_r$ is a Poisson kernel.
\end{corollary}
Note that this corollary holds for $r = 1$ if we replace 
$\frac{1}{2\pi}\int_0^{2\pi} |u(e^{i\theta}) - u(0)| \, d\theta$ by 
$\|u - u(0)\|_{h^1_{\mathbb{R}}}$. 
\begin{proof}
The rightmost equality is proved in Theorem \ref{thm:szop-h1-r}.
Also note that the inequality holds for $r = 0$ trivially.  
For other values of $r$, 
if $u$ is the real part of a function in $H^1$, then we can write 
$u(re^{i\theta}) = P_{\theta}(r e^{i \cdot})*f(e^{i\theta})$, 
where $f$ is the boundary 
value function of $u$. The result then follows from the above lemma by 
letting $P = P_r$. 

 If $u$ is not the real part of an $H^1$ function, 
then $u$ is still in 
$h^1_{\mathbb{R}}$ since it is nonnegative (see \cite[Theorem 1.1]{D_Hp}). 
Let $0 < s < 1$. 
Define $u_s$ by 
$u_s(z) = u(sz)$. Then $u_s$ is the real part of 
an $H^1$ function, since it is actually continuous in $\mathbb{D}$.  
So for fixed $r$, the above inequality is true for $u_s$.  (Note that 
$u_s(0) = u(0)$).  If we let $s \rightarrow 1$, we get the result for 
$u$, since $u(rse^{it}) \rightarrow u(re^{it})$ as $s \rightarrow 1$ 
uniformly for 
$t \in [0,2\pi)$. 
\end{proof}

We must now deal with functions that are 
allowed to be negative. 
\begin{lemma}\label{lemma:negative-conv-bound}
Suppose that $P$ 
is a nonnegative 
integrable function on 
$[\alpha,\beta]$ and $f$ is an integrable %
function such that $\|f\|_1 = 1$, 
where $\|\cdot\|_1$ denotes the $L^1$ norm with normalized Lebesgue measure.   
Let $P^*$ denote the decreasing rearrangement of 
$P$, so that 
$P^*(t) = \inf \{x: m(\{y \in [\alpha,\beta]: 
 P(y) > x\}) \leq t\}$. 
Then
$\|P * f - P * f * 1\|_1 \leq \|Q - Q * 1\|_1 $ 
where $Q$ is some function of the form 
$aP^*(x) - bP^*(\alpha + \beta -x)$, where $a + b = 1$. 
\end{lemma}
\begin{proof}
Without loss of generality we may assume that $\alpha=0$ and 
$\beta=1$.  

We may assume that $P$ and $f$ are %
continuous (and thus uniformly continuous), 
since these functions are 
dense in $L^1([0,1])$.  
Let $\epsilon > 0$ and assume without loss of generality that 
$\epsilon < 1$.  
Let $M = \|f\|_\infty + \|P\|_\infty + 1$.  

Approximate $P$ and $f$ by the step functions
\[
\begin{split}
\widetilde{P} &= \sum_{k=0}^{n-1} c_k \chi_{[k/n,(k+1)/n)} \text{ and }\\ 
\widetilde{f} &= \sum_{k=0}^{n-1} d_k \chi_{[k/n,(k+1)/n)}
\end{split}
\] respectively, 
so that \[
\int_{k/n}^{(k+1)/n} \wt{f} \, dx = \int_{k/n}^{(k+1)/n} f \, dx 
\qquad \text{and} \qquad
\int_{k/n}^{(k+1)/n} \wt{P} \, dx = \int_{k/n}^{(k+1)/n} P \, dx
\]
for each $0 \leq k \leq n-1$.  
We may choose $n$ large enough so that 
$|f(x) - f(y)| < \epsilon$ if $|x-y| \leq 1/n$ and   
$\|P - \wt{P}\|_\infty < \epsilon$ and
$\|f - \wt{f}\|_\infty < \epsilon$.  
Thus 
$\|P * f - \wt{P} * \wt{f}\|_\infty < M\epsilon$.  
Also $\|P^* - \wt{P}^{*}\|_{\infty} < \epsilon$ 
because decreasing rearrangement decreases the $L^\infty$ distance 
between two functions.

For $k$ not in $[0, n-1]$, define $c_k$ and $d_k$ so the sequences 
$\{c_k\}_{k \in \mathbb{Z}}$ and $\{d_k\}_{k \in \mathbb{Z}}$ are periodic with 
period $n$. 
Let 
$A$ be the matrix with $(i,j)$ entry 
$a_{i,j} = c_{j-i} d_i$. 
Then the column sums 
$C_j$ of $A$ are equal to $n \wt{P} * \wt{f} (j/n)$.  Define
$D_j = \max(C_j - n, 0)$ and let $D$ be the sum of the $D_j$.

Now, let $A'$ be the matrix $A$ with rows rearranged so that they 
are in decreasing order.  Define $C_j'$ and $D_j'$ similarly.  
Then $C_j' = n \wt{Q}(j/n)$, where 
$\wt{Q}$ has the form $a\wt{P}^*(x) - b\wt{P}^*(1-x)$, where
$\wt{P}^*$ is the decreasing rearrangement 
of $\wt{P}$ and $a+b = 1$.  Define 
$Q(x) = aP^*(x) + b P^*(1-x)$. Then 
$\|Q - \wt{Q}\| \leq \epsilon$.

Because $|c_{j} - c_{j+1}| < \epsilon$ and 
$|d_{j} - d_{j+1}| < \epsilon$ we have 
$|\wt{Q}(j/n + \delta) - \wt{Q}(j/n)| < \epsilon$ and 
$|\wt{P} * \wt{f}(j/n + \delta) - \wt{P}*\wt{f}(j/n)| < M \epsilon$ 
for $|\delta| < 1/n$.  

Let $\mu = a - b$.  
Thus 
\[\begin{split}
| 
\max(\wt{P} * \wt{f}(j/n)-\mu,0) - 
\max(\wt{P} * \wt{f}(j/n + \delta) - \mu,0) |
              &< M \epsilon \text{ and} \\
| \max(\wt{Q}(j/n)-\mu,0) - 
\max(\wt{Q}(j/n + \delta) - \mu,0) | 
&< \epsilon.
\end{split}\] 
Integrating each of these expressions over $\delta \in [0, 1/n]$ and 
summing from $j = 0$ to $n-1$ gives that 
\[\left|D/n - \int_{\{x:\wt{P} * \wt{f}(x) > \mu\}} \wt{P} * \wt{f} - 
\mu \, dx\right| < \epsilon \] 
and
\[
\left|D'/n - \int_{\{x:\wt{Q}(x) > \mu\}} \wt{Q}(x) - \mu \, dx\right| < M \epsilon.\]

Theorem 
\ref{thm:matrix-sum-minus-mu} with 
$\mu = a-b$ implies that 
$D \leq D'$. Thus
\[
\int_{\{x:\wt{P} * \wt{f}(x) > \mu\}} \wt{P} * \wt{f} - \mu \, dx   \leq
(M + 1)\epsilon +
 \int_{\{x:\wt{Q}(x) > \mu\}} \wt{Q}(x) - \mu \, dx.
\]
and therefore
\[
\int_{\{x:P * f(x) > \mu\}} P * f - \mu \, dx   \leq
2(M + 1)\epsilon +
 \int_{\{x:Q(x) > \mu\}} Q(x) - \mu \, dx.
\]

Since this is true for any $\epsilon$, we must have 
\[
\int_{\{x:P * f(x) > \mu\}} P * f - \mu \, dx   \leq
 \int_{\{x:Q(x) > \mu\}} Q(x) - \mu \, dx.
\]

The result now follows from Lemma 
\ref{lemma:alt-f-minus-mu}.
\end{proof}

It is useful to have a condition under which we can conclude that 
$Q(x) = P^*(x)$.  The following theorem provides such a condition. 

\begin{theorem} \label{thm:original-greater-difference}
Suppose that $P$  
is a continuous function on 
$[\alpha,\beta]$ and that $P$ is nonnegative with average 
$1$ and 
decreasing.  
Let 
$Q(x) = a P(x) - b P(\beta+ \alpha - x)$ 
for some nonnegative real 
numbers $a$ and $b$ such that $a + b = 1$. 
Then 
there is a $c$ between 
$\alpha$ and $\beta$ such that  
$a P(c) - b P(\beta + \alpha - c) = a - b$.  If for some such $c$, 
\[
\int_\alpha^{\alpha + c} P \, dx + \int_{\beta-c}^{\beta} P \, dx \geq 2c,
\]
then
$\|P  - P * 1\|_1 \geq \|Q - Q * 1\|_1 $, where 
$\|\cdot\|_1$ denotes the $L^1$ norm with normalized Lebesgue measure. 
\end{theorem}

\begin{proof}
We will suppose without loss of generality 
that $\alpha = 0$ and 
$\beta = 1$.
Also we may suppose without loss of generality that $a \geq b$, since 
if $Q(x) = a P(x) - b P(1-x)$ and 
$\widehat{Q}(x) = b P(x) - a P(1-x)$, then 
$\widehat{Q}(x) = - Q(1-x)$, so 
$\|Q - Q * 1\|_1 = \|\widehat{Q} - \widehat{Q}*1\|_1$. 
The average of $aP(x) -bP(1-x)$ is 
$a - b = 1 - 2b$.  
Let $c$ be such that 
 $aP(c) - bP(1-c) = 1 - 2b$.  
(Such a $c$ exists by the integral mean 
value theorem). 

Suppose that $f(x)$ is nonnegative between 
$0$ and $c$, and that $\nu$ is nonnegative. 
Let $A(\nu, f; c)$ denote the area below 
$f$ and above $\nu$ and between $0$ and $c$.  
Let $\widetilde{A}(\nu, f; c)$ be $0$ if 
$f(c) \leq \nu$ and be $f(c) - \nu$ if 
$f(c) \geq \nu$.  Then 
$A(\nu, f; c) = \int_0^c \widetilde{A}(\nu, f; x) \, dx$. 

Let $\mu = a - b$. 
Note that between $0$ and $c$, the 
function $Q$ always has a value of at least 
$\mu$, so 
$A(\mu, Q; c) = \int_0^c Q(x) - \mu \, dx$. 
Similarly, we have that 
$A(\mu, P; c) = \int_0^c P(x) - \mu \, dx$. 
Also $P(x) = Q(x) + b P(x) + b P(1-x)$. 
Thus 
$A(\mu, P; c) = A(\mu, Q; c) + b \int_0^c P \, dx + b \int_{1-c}^1 P \, dx$. 

If $P(x) \geq 1$ observe that $P(x) - 1 = P(x) - \mu - 2b$ 
since $\mu = 1- 2b$. 
So in this case 
$\widetilde{A}(1, P; c) = \widetilde{A}(\mu, P; c) - 2b$. 
If however $P(x) \leq 1$ then 
$\widetilde{A}(1 ,P; c) = 0$ and 
$\widetilde{A}(\mu, P; c) \leq 2b$ since 
$1 - \mu = 2b$.  Thus in either case 
$\widetilde{A}(1, P; c) \geq \widetilde{A}(\mu, P; c) - 2b$. 
This implies that 
$A(1, P; c) \geq A(\mu, P; c) - 2bc$.  And thus 
\[
A(1,P;c) \geq A(\mu, Q;c) + b \int_0^c P \, dx + b \int_{1-c}^1 P \, dx - 2bc.
\]
So we will have $A(1,P;c) \geq A(\mu, Q; c)$ if 
\[
\int_0^c P \, dx + \int_{1-c}^1 P \, dx \geq 2c.
\]
In this case, by Lemma 
\ref{lemma:alt-f-minus-mu}, 
\[
\int_0^1 |P(x) - 1| \, dx = 2A(1,P;1) 
\geq 2A(1,P;c) \geq 2A(\mu, Q; c).
\]  
Since $Q(c) = \mu$ and $Q(x) \leq \mu$ for 
$x > c$, we have that 
$2A(\mu, Q; c) = 2A(\mu, Q; 1) = 
 \int_0^1 |Q(x) - \mu|\, dx$. 
\end{proof}

We are now able to prove the following theorem. 
\begin{theorem}\label{thm:szop-h1-r}
Suppose that $0 \leq r < 1$ and $u \in h^1_{\mathbb{R}}$.  Then 
\[
\begin{split}
\frac{1}{2\pi}\int_0^{2\pi} |u(re^{i\theta}) - u(0)| \, d\theta
&\leq 
\inf_{a \in \mathbb{R}} \|u-a\|_{h^1_{\mathbb{R}}} 
\frac{1}{2\pi} \int_0^{2\pi} |P_r(e^{i\theta}) - 1| \, d\theta \\
&= \inf_{a \in \mathbb{R}}\|u-a\|_{h^1_{\mathbb{R}}}\left(2 - \frac{4}{\pi}\arccos(r) \right)
\end{split}
\]
where $P_r$ is a Poisson kernel.
\end{theorem}
Note that the theorem holds for $r=1$ if we replace 
$\frac{1}{2\pi} \int_0^{2\pi} |u(e^{i\theta}) - u(0)| \, d\theta$ by 
$\|u - u(0)\|_{h^1_{\mathbb{R}}}$. 
\begin{proof}
First note that it suffices to prove the statement for $a = 0$, since 
$u(re^{i\theta}) - u(0) = \widetilde{u}(re^{i\theta}) - \widetilde{u}(0)$ 
if $\widetilde{u} = u - a$. 

First assume that $u$ is the real part of an $H^1$ function, and let 
$f$ be its boundary value function. 
Let $P(\theta)$ be the 
Poisson kernel $P_r(e^{i \theta})$
restricted to $0 \leq \theta \leq \pi$.  
Let $a, b \geq 0$ and $a + b = 1$. 

Let $c$ be the number between $0$ and $\pi$ such that 
$a P(c) - b P(\pi - c) = a - b$.  
Such a $c$ exists by the integral mean value theorem. 
Note that $c$ is 
unique and a 
continuous function of $a$ by the 
fact that $aP(c) - b P(\pi - c)$ is strictly 
decreasing and the implicit function theorem. 
Note that $P(\pi / 2) = (1-r^2)/(1+r^2) < 1$.
If $a = 1/2$, then $c = \pi/2$ since then 
$a P(\pi/2) - b P(\pi - (\pi/2)) = 0$. 
However, 
for $0 \leq a < 1/2$, the number $c$ is strictly less than
$\pi/2$, because 
\[a P(\pi/2) - b P(\pi - (\pi/2)) = 
(a-b) P(\pi/2) < (a-b).\] 

Let 
\[ 
\alpha = \arg[(e^{ic}-r)/(1-r)] \]
and
\[
\beta = \arg[(-1-r)/(e^{i(\pi-c)}-r)] = \arg[(e^{ic}+r)/(1+r)].
\]
Now $\frac{1}{2\pi}\int_\gamma^\delta P(x) \, dx$ is 
equal to the harmonic measure of the arc of the unit circle
$[e^{i\gamma}, e^{i \delta}]$ at the point $r$, 
which equals $\phi/\pi - (\delta - \gamma)/(2\pi)$, where 
$\phi$ is the angle subtended at $z$ by the arc 
(see \cite[Chapter 1, Exercise 1]{Garnett-Marshall_Harmonic-Measure}). 
Thus
$\frac{1}{2\pi}\int_0^c P(x) \, dx = \alpha/\pi - c/(2\pi)$. 
Also
$\frac{1}{2\pi}\int_{\pi-c}^\pi P(x) \, dx = \beta/\pi - c/(2\pi)$. 
So the sum of the last two integrals is $(\alpha + \beta - c)/\pi$.  
We will show that this is
at least $2c/2\pi$.  To do so we need $\alpha + \beta \geq 2c$. 
But 
\[
\frac{e^{ic}-r}{1-r} 
\frac{e^{ic}+r}{1+r} =
\frac{e^{2ic}-r^2}{1-r^2}.
\]
The argument of the last expression is measure of the angle with 
vertex $r^2$ and endpoints $0$ and $2c$, which is at least $2c$ if 
$r > 0$ and $0 \leq c \leq \pi/2$.  
But we have shown above that $0 \leq c \leq \pi/2$. 
So we always have 
$\int_0^c P(x) \, dx + \int_{\pi - c}^{\pi} P(x) \, dx \geq 2c$. 

Let $\|\cdot\|$ denote the $L^1$ norm with normalized 
Lebesgue measure. 
 Apply Theorem 
\ref{thm:original-greater-difference} to see that 
$\|Q - Q * 1\| \leq \|P  - P * 1\|$ for any $Q$ of the form 
$a P(x) - b P(\pi - x)$ where $a,b \geq 0$ and $a + b = 1$. 
If $\widetilde{P}(x)$ is defined on 
$[0, 2\pi]$ by $\widetilde{P}(x) = P(x/2)$, 
and similarly for $\widetilde{Q}$, then this implies that 
$\|\widetilde{Q}  - \widetilde{Q} * 1\| \leq 
\|\widetilde{P} - \widetilde{P} * 1\|$.

Now let $P^*$ be the decreasing rearrangement of 
$P_r(e^{i\theta})$, thought of as a function of 
$\theta$, where $0 \leq \theta \leq 2\pi$.  Notice that 
Lemma \ref{lemma:negative-conv-bound} shows that 
$\|P_r*f-P_r*f*1\| \leq \|\widehat{Q} - \widehat{Q}*1\|$, where 
$\widehat{Q}$ is some function of the form 
$aP_r^*(x) - bP_r^*(2\pi - x)$
and 
$f$ is any continuous function 
on $[0, 2\pi]$ that has norm $1$. 
But for $0 \leq x \leq 2 \pi$ one has 
$P_r^*(x) = P_r(x/2) = \widetilde{P}(x)$ 
since $P_r$ is symmetric about $0$.  And thus we have 
shown above that $\|\widehat{Q} - \widehat{Q}*1\|$ is at most 
$\|\widetilde{P}  - \widetilde{P} * 1\|$.
But $\widetilde{P}*1 = P_r * 1 = 1$, where 
$P_r$ is considered as a function on $[0,2\pi]$. And also 
$P_r - 1$ is equimeasurable with 
$\widetilde{P}-1$.  And thus 
$\|P_r*f - P_r*f*1\| \leq \|P_r - P_r*1\|$. 

Now notice that $P_r(\theta) = 1$ if 
$\theta = \pm \arccos(r)$.  
Now 
\[\frac{1}{2\pi} \int_{-\arccos(r)}^{\arccos(r)} 
P_r(\theta) \, d\theta 
= \alpha/\pi - 2\arccos(r)/(2\pi),
\] where 
$\alpha$ is the measure of the angle between 
$e^{-i \arccos(r)}$, $r$, and $e^{i \arccos(r)}$.  But the 
measure of the angle is $\pi$.  
Thus the integral is $1-\arccos(r)/\pi$.  
The (normalized) integral of $P_r$ over the complimentary interval is 
thus $\arccos(r)/\pi$.

Thus 
\[
\begin{split}
\frac{1}{2\pi}
\int_{-\pi}^{\pi} |P_r(\theta) - 1| \, d\theta &= 
\left[ \left(1 - \frac{\arccos(r)}{\pi}\right) - \frac{\arccos{r}}{\pi}\right] +\\ 
& \qquad \left[ 
-\frac{\arccos(r)}{\pi} + \left(1 - \frac{\arccos(r)}{\pi}\right)\right] \\ &= 2 - \frac{4}{\pi} \arccos(r) .
\end{split}
\]
This proves the result if $u$ is the real part of an $H^1$ function. 

Now suppose that $u$ is not the real part of an $H^1$ function.  As before, 
let $u_s$ be defined by $u_s(z) = u(sz)$ for $0 < s < 1$.  Then 
$\|u_s\|_{h^1_{\mathbb{R}}} \leq \|u\|_{h^1_{\mathbb{R}}}$ since the $M_1$ integral means increase 
for harmonic functions (see \cite{D_Hp}).  So 
\[
\frac{1}{2\pi}\int_0^{2\pi} |u_s(re^{i\theta}) - u(0)| \, d\theta
\leq \|u\|_{h^1_{\mathbb{R}}} 
\left(2 - \frac{4}{\pi}\arccos(r) \right).
\]
Letting $s \rightarrow 1$ gives the result. 

\end{proof}

\begin{corollary}
The value of $\|\szop\|_{h^1_{\mathbb{R}} \rightarrow a^1_{\mathbb{R}}} = 1 $.
\end{corollary}
\begin{proof}
We have that 
\[
\|\szop\|_{h^1_{\mathbb{R}} \rightarrow a^1_{\mathbb{R}}} \leq 
\int_0^1 \|\szop_r\| \, 2r \, dr = 
\int_0^1 \left(2 - \frac{4}{\pi}\arccos(r)\right) \, 2r \, dr 
= 1.
\]
This is attained for the Poisson kernel, though of as the function
$re^{i\theta} \mapsto P_{r}(e^{i\theta})$ defined in the unit disc. 
\end{proof}

\providecommand{\bysame}{\leavevmode\hbox to3em{\hrulefill}\thinspace}
\providecommand{\MR}{\relax\ifhmode\unskip\space\fi MR }
\providecommand{\MRhref}[2]{%
  \href{http://www.ams.org/mathscinet-getitem?mr=#1}{#2}
}
\providecommand{\href}[2]{#2}

\end{document}